\documentclass[a4paper,12pt]{amsart}
\usepackage{amssymb}
\usepackage{amsmath}
\usepackage{ifthen}
\usepackage{graphicx,float}
\usepackage{mathtools}
\usepackage{color}
\usepackage{latexsym, amsthm, amscd, euscript, float, subfig}
\newtheorem{theorem}{Theorem}[section]
\newtheorem{lemma}[theorem]{Lemma}
\newtheorem{corollary}[theorem]{Corollary}

\makeatletter
\newcommand*{\rom}[1]{\expandafter\@slowromancap\romannumeral #1@}
\makeatother

\theoremstyle{definition}

\newtheorem{remark}[theorem]{Remark}

\newcounter{minutes}
\divide\time by 60
\newcounter{hours}
\multiply\time by 60
\addtocounter{minutes}{-\time}

\numberwithin{equation}{section}
\begin{document}
	
\bibliographystyle{amsplain}
	
\title[Geometric properties]
{Geometric properties of certain integral operators involving Hornich operations}
	
\def\thefootnote{}
\footnotetext{ \texttt{\tiny File:~\jobname .tex,
	printed: \number\day-\number\month-\number\year,
\thehours.\ifnum\theminutes<10{0}\fi\theminutes}
} \makeatletter\def\thefootnote{\@arabic\c@footnote}\makeatother
\author[Shankey Kumar]{Shankey Kumar}
\address{Shankey Kumar, School of Mathematical Sciences, National Institute of Science Education and Research Bhubaneswar, Khurda 752050, Odisha, India.}
\email{shankeygarg93@gmail.com, shankeygarg22@niser.ac.in}

	
\begin{abstract}
In this article, we investigate some standard  geometric properties of the  integral operators
$$
J_\alpha [f](z)= \int_{0}^{z}\bigg(\frac{f(w)}{w}\bigg)^\alpha dw, \,\,\, \alpha \in \mathbb{C} \text{ and } |z|<1,
$$
and
$$
I_\beta [g](z)= \int_{0}^{z}\big(g'(w)\big)^\beta dw, \,\,\, \beta \in \mathbb{C} \text{ and } |z|<1,
$$
where $f$ and $g$ are elements of certain classical families of normalized analytic functions defined on the unit disk. 
In particular,  preserving properties of the Hornich sum of the operators $J_\alpha$ and $I_\beta$ will be studied. Moreover, we also present sharp pre-Schwarzian norm estimate of such integrals.  
\end{abstract}
	
\maketitle
\section{Introduction}
Let $\mathcal{A}$ be the class of functions $f (z) = z+ a_2 z^2+ \dots$, analytic and normalized in the unit disk $\mathbb{D}:=\{z\in\mathbb{C}:|z|<1\}$. We denote by $\mathfrak{F}$ the class of all those analytic functions $f\in\mathcal{A}$ satisfying $f'(z)\neq 0, z\in \mathbb{D}$. The subclass of $\mathfrak{F}$ consisting of all univalent functions is denoted by $\mathcal{S}$. Here we consider some classical subclasses of $\mathcal{S}$.
	
If a function $f\in \mathcal{S}$ maps $\mathbb{D}$ onto a convex domain then $f$ is called a convex function. The notation $\mathcal{K}$ is defined for the class of convex functions. A function $f\in \mathcal{S}$ is said to be close-to-convex if there is a function $g\in\mathcal{K}$ and a real number $\alpha\in(-\pi/2,\pi/2)$ such that
$$
{\rm Re}\,\bigg(e^{i\alpha}\frac{f'(z)}{g'(z)}\bigg)>0, \,\, z\in\mathbb{D},
$$
see \cite[Vol. 2, p.~2]{GMBook}. The notation $\mathcal{C}$ stands for the class of close-to-convex functions. By the definition, it is clear that $\mathcal{K}\subsetneq \mathcal{C}$. In 1952, Kaplan in \cite{Kaplan52} proved that a function $f\in\mathfrak{F}$ is close-to-convex if and only if
\begin{equation}\label{4eq1.1}
\int_{\theta_1}^{\theta_2}  {\rm Re}\,\bigg(\frac{zf''(z)}{f'(z)}+1\bigg)d\theta > -\pi, \,\,\ z=re^{i\theta},
\end{equation}
for each $0<r<1$ and for each pair of real numbers $\theta_1$ and $\theta_2$ with $\theta_1<\theta_2$; see \cite{Duren83,GMBook,MM-Book} for more information. In this sequence, we have some important subclasses of $\mathfrak{F}$, which were widely used by many authors for different prospective.

\subsection{The class $\mathcal{G}(\gamma)$} A function $f\in \mathfrak{F}$ belongs to the class $\mathcal{G}(\gamma)$, $\gamma>0$, if it satisfies the condition 
$$
1+\frac{zf''(z)}{f'(z)}<1+\frac{\gamma}{2}, \,\,\ z\in\mathbb{D}.
$$
The class $\mathcal{G}:=\mathcal{G}(1)$ was first introduced by Ozaki \cite{Ozaki41} and proved the inclusion relation $\mathcal{G}\subset \mathcal{S}$. Also, Umezawa \cite{Ume52} studied this class and showed that this class contains the class of functions convex in one direction.
The Taylor coefficient problem for the class $\mathcal{G}(\gamma)$, $0<\gamma\leq 1$, is discussed in \cite{Obradovic13}. Recently, the radius of convexity of the functions in the class $\mathcal{G}(\gamma)$, $\gamma>0$, is obtained in \cite{Shankey20}. More information about the class $\mathcal{G}(\gamma)$, $\gamma>0$, can be found in \cite{ponnusamy95,ponnusamy96,ponnusamy07}.

\subsection{The class $\mathcal{S}_{\theta}(\gamma)$}
We consider the class
\begin{equation}\label{4eq1.1}
\mathcal{S}_{\theta}(\gamma)=\Bigg\lbrace f\in\mathcal{A}:{\rm Re}\,\bigg( e^{i\theta} \frac{zf'(z)}{f(z)}\bigg)<\Big(1+\frac{\gamma}{2}\Big) \cos \theta, \,\,\ z\in\mathbb{D}\Bigg\rbrace
\end{equation}
where $\gamma>0$ and $-\pi/2<\theta<\pi/2$. The class $\mathcal{S}_{\theta}(\gamma)$ is a non-empty set since the function $f(z)=z$ satisfy the condition \eqref{4eq1.1}. Also, we introduce a non-trivial example in the final section. It is easy to observe that if $g(z)=zf'(z)$ then $f\in \mathcal{G(\gamma)}$ if and only if $g \in \mathcal{S}_{0}(\gamma)$ for every $\gamma>0$. For $\theta=0$ and $\gamma=1$, the class $\mathcal{S}_{\theta}(\gamma)$ is introduced in \cite{Shah71}, and recently studied in \cite{Shankey20}. In \cite{Shah71}, Shah pointed out that the function of the class $\mathcal{S}_0(1)$ is starlike in one direction. Robertson \cite{Robertson36} has proved that if $g(z)=z+\sum_{n=2}^{\infty}a_n z^n$ is starlike in one direction for $|z|<1$ than $|a_n|\leq n^2$.
We set $\mathcal{S}(\gamma):=\bigcup_\theta\,\mathcal{S}_{\theta}(\gamma)$.
	
\subsection{The class $\mathcal{K}(\lambda)$} Consider the class $\mathcal{K}(\lambda)$, $\lambda<1$, which is expressed as
$$
\mathcal{K}(\lambda)=\bigg\lbrace f\in \mathfrak{F}:{\rm Re}\,\bigg(1+\frac{zf''(z)}{f'(z)}\bigg)>\lambda, \,\,\ z\in\mathbb{D}\bigg\rbrace.
$$
It is evident that $\mathcal{K}=\mathcal{K}(0)$.
By using \eqref{4eq1.1} we can easily observe that if $f\in\mathcal{K}(-1/2)$ then $f\in \mathcal{C}$. Note that the class $\mathcal{K}(\lambda)$, $-1/2\leq \lambda<1$, is introduced, for instance, in \cite{LPQ16} (see also \cite{Arora18,Ali18}).
In \cite{PSY14}, Ponnusamy et al. obtained that every section of a function in the class $\mathcal{K}(-1/2)$ is convex in the disk
$|z| < 1/6$, and the quantity $1/6$ is best possible. Further, radius of convexity for functions in the class $\mathcal{K}(\lambda)$, $\lambda<1$, calculated in \cite{Shankey20}. This class has been considered by several authors on different counts (see \cite{ADRS15,D13,G81,Shankey19} and references therein).
	
\medskip	
Now, we discuss some familiar operations in geometric function theory. One of them is the Hornich sum of functions $f,g\in\mathfrak{F}$ defined as
$$
(f\oplus g)(z)=\int_{0}^{z}f'(w)g'(w)dw.
$$  
Another one is the Hornich scalar multiplication operation (or operator)
$$
I_\beta[g](z):=(\beta \star g)(z)=\int_{0}^{z} (g'(w))^\beta dw, 
$$
where $f\in \mathfrak{F}$, $\beta \in \mathbb{C}$ and $|z|<1$. Here, the choice for the branch of $(g'(w))^\beta$ has been taken in such a way that $(g'(0))^\beta=1$. It clearly follows that $I_\alpha I_\beta = I_{\alpha \beta}$. It is easy to check that the class $\mathfrak{F}$ forms a vector space over $\mathbb{C}$ under the Hornich operations
(the Hornich sum and the  Hornich scalar multiplication operation). 
The Hornich operations can be used to obtain many important integral operators. In the sequel, the following definition due to Kim and Merkes \cite{Kim72} is useful for our main results:
\begin{equation}\label{AF}
	A(\mathfrak{F})=\{\alpha \in \mathbb{C}: I_\alpha(\mathfrak{F})\subset \mathcal{S}\},
\end{equation}
here the notation $I_\alpha(\mathfrak{F})$ is defined by
\begin{equation}\label{IgammaF}
	I_\alpha(\mathfrak{F})=\{I_\alpha[f]:\,f\in \mathfrak{F}\}.
\end{equation}
Recall that the inclusion $\{\alpha:\,|\alpha|\le 1/2\}\subset A(\mathcal{K})$ was first proved by 
Singh and Chichra in \cite{Singh77}. 
Further, the inclusion $[0,3/2]\subset A(\mathcal{K})$ was due to  Nunokawa \cite{Nunokawa69}. In continuation to this analysis, Merkes \cite{Merkes85} proposed the conjecture that $\{\alpha\in\mathbb{C}:\,|\alpha-1|\le 1/2\}\subset A(\mathcal{K})$.
However, Aksent'ev and Nezhmetdinov
\cite{Aksent'ev82} disproved the conjecture of Merkes by showing that 
\begin{equation}\label{AK}
A(\mathcal{K})=\{\alpha\in\mathbb{C}:\,|\alpha|\le 1/2\}\cup[1/2,3/2]
\end{equation}
(see also \cite{Kimponnusamy04}).
Recently, Kumar and Sahoo \cite{Shankey19} extend this result and obtained that
$$
A(\mathcal{K(\lambda)})=\left\{\alpha\in\mathbb{C}:\,|\alpha|\le \frac{1}{2(1-\lambda)}\right\}\bigcup \left[\frac{1}{2(1-\lambda)},\frac{3}{2(1-\lambda)}\right],\, \mbox{ for } \lambda<1.
$$ 
Now if we put $\lambda=-1/2$ in the above set then we have 
$$
A(\mathcal{K}(-1/2))=\left\{\alpha\in\mathbb{C}:\,|\alpha|\le \frac{1}{3}\right\}\bigcup\Big[\frac{1}{3},1\Big].
$$
It is here appropriate to recall that,
in one hand, due to Pfaltzgraff as shown in \cite[Corollary~1]{Pfa75} 
$I_\alpha(\mathcal{S})\subset \mathcal{S}$ for $|\alpha|\le 1/4$. On the other hand, Royster proved in \cite[Theorem~2]{Royster65} that for each number $\alpha\neq 1$ with $|\alpha|>1/3$,
there exists a function $f\in\mathcal{S}$ such that $I_\alpha[f]\not\in\mathcal{S}$ (see also \cite{Ali18,Kim04,KS06}).
Note that the description of the whole set $A(\mathcal{S})$ is still open.

In 1974, Kim and Merkes \cite{Kim74} studied the operator 
$$
I_{\alpha,\beta}[f,g](z):= \big( I_\alpha[f] \oplus I_\beta [g]\big)(z)=\int_{0}^{z}(f'(w))^\alpha (g'(w))^\beta dw, \,\,\, \alpha,\beta \in \mathbb{R} \text{ and } |z|<1,
$$
defined on $f,g\in\mathfrak{F}$. By the definition of the operator $I_{\alpha,\beta}$ it is clear that this is a combination of the Hornich operations. One of the interesting results obtained in \cite{Kim74} for the operator $I_{\alpha,\beta}$ is the following:
	
\medskip
\noindent  
{\bf Theorem A.} {\em Let $f,g\in \mathcal{K}$. For the real numbers $\alpha$ and $\beta$, we have 
\begin{enumerate}
\item[\bf (i)] 
$I_{\alpha,\beta}[f,g] \in \mathcal{K}$ if and only if $\alpha\geq0,\beta\geq 0, \alpha+\beta\leq1$. 
\item[\bf (ii)] 
$ I_{\alpha,\beta}[f,g] \in \mathcal{C}$ if and only if $-1/2\leq\alpha,\beta\leq 3/2 , -1/2\leq\alpha+\beta\leq3/2$.
\end{enumerate}}
	
\medskip
Theorem~A(i) says that if there exist positive $\alpha$ and $\beta$ satisfying $\alpha+\beta>1$ or at least one of them is negative, then $I_{\alpha,\beta}[f,g]$ is no more in $\mathcal{K}.$ This means that if we replace the term ``{\em if and only if}" with ``{\em if}" in  Theorem~A(i), then the result would be called sharp. Same concept is applied for similar other results. 
Further, Theorem A has been extended in \cite{Ali18}
by replacing $\mathcal{K}$ with 
$\mathcal{K(\lambda)}$, $-1/2\leq\lambda<1$, and $\mathcal{G}(\gamma)$, $0<\gamma\leq1$, separately.
	
In 1972, Kim and Merkes \cite{Kim72} considered the nonlinear operator $J_\alpha[f]$, $f\in\mathfrak{F}$ such that $f(0)=0$, defined by 
$$
J_\alpha[f](z):=(I_\alpha[J[f]])(z)=\int_{0}^{z}\bigg(\frac{f(w)}{w}\bigg)^\alpha dw, \,\,\, \alpha \in \mathbb{C} \text{ and } |z|<1,	
$$
and they showed that $J_\alpha(\mathcal{S})\subset \mathcal{S}$ for $|\alpha|\leq 1/4$, i.e. $A(J(\mathcal{S}))=\{\alpha\in\mathbb{C}:\,|\alpha|\le 1/4\}$.
If $\alpha=1$, then one has the well-known Alexander transformation $J[f]:=J_1[f]$ which is defined by
$$
J[f](z)=\int_{0}^{z}\frac{f(w)}{w} dw, \,\,\,  |z|<1.	
$$
We know that the class $\mathcal{S}$ does not preserve by the Alexander transform, see \cite[$\S8.4$]{Duren83}. This motivates us to study 
the classical classes of functions under
the Alexander and related transforms considered 
in this paper.
We use the following notation concerning the Alexander operator $J$: 
\begin{equation}\label{JF}
	J(\mathfrak{F})=\{J[f]:\,f\in\mathfrak{F}\}.
\end{equation}
By the definitions \eqref{AF} and \eqref{IgammaF} we formulate
$$
A\big(J(\mathfrak{F})\big)=\{\alpha \in \mathbb{C}: J_\alpha(\mathfrak{F})\subset \mathcal{S}\}
~~\mbox{ and }~~
J_\alpha(\mathfrak{F})=(I_\alpha\circ J)(\mathfrak{F}).
$$ 
For the starlike family $\mathcal{S}^*$, Singh and Chichra in \cite{Singh77} proved that $A(J(\mathcal{S}^*))\supset \{\alpha\in\mathbb{C}:\,|\alpha|\leq 1/2\}$. 
However, as noted in \eqref{AK}, the complete range of $\alpha$ for $A(J(\mathcal{S}^*))$ was found by Aksent'ev and Nezhmetdinov \cite{Aksent'ev82}, since 
$J(\mathcal{S}^*)=\mathcal{K}$. More interestingly, for a given $\tau>0$,
Kim et al. \cite{Kim04} could generate a subclass $\mathcal{F}$ of $\mathcal{A}$ such that   
$J_\alpha(\mathcal{F})\subset \mathcal{S}$ for all $\alpha\in\mathbb{C}$ with $|\alpha|\leq\tau$.

	In \cite{Kim74}, authors also studied the operator, for $f,g\in\mathfrak{F}$,
	$$
	J_{\alpha,\beta}[f,g](z):=\big( J_\alpha[f] \oplus J_\beta [g]\big)(z)=\int_{0}^{z}\bigg(\frac{f(w)}{w}\bigg)^\alpha \bigg(\frac{g(w)}{w}\bigg)^\beta dw, \,\,\, \alpha,\beta \in \mathbb{R} \text{ and } |z|<1.
	$$
	This can be easily generated with the help of the Alexander transformation and the Hornich operations. Corresponding to the operator $J_{\alpha,\beta}$ they have the following result:
	
	\medskip
	\noindent
	{\bf Theorem B.} {\em Let $f,g\in \mathcal{K}$. For the real quantities $\alpha$ and $\beta$, we have 
		\begin{enumerate}
			\item[\bf (i)] $ J_{\alpha,\beta}[f,g] \in \mathcal{K}$ if and only if $\alpha\geq0,\beta\geq 0, \alpha+\beta\leq2$. 
			\item[\bf (ii)]
			$ J_{\alpha,\beta}[f,g] \in \mathcal{C}$ if and only if $-1\leq\alpha,\beta\leq 3 , -1\leq\alpha+\beta\leq3$.
	\end{enumerate}}
	
	\medskip
	Recently, the sharp radii of convexity for the integral operator
	$$
	C_{\alpha,\beta}[f,g](z)=\int_{0}^{z}\bigg(\frac{f(w)}{w}\bigg)^\alpha (g'(w))^\beta d w, \,\,\, \alpha,\beta \in \mathbb{R} \text{ and } |z|<1,
	$$
	over subclasses of the class $\mathfrak{F}$ investigated in \cite{Shankey20}.
	This operator can be obtained by replacing $f'(w)$ with $(J[f])'(w)$ in $I_{\alpha,\beta}[f,g]$ or $(J[g])'(w)$ with $g'(w)$ in $J_{\alpha,\beta}[f,g]$. Here, we choose branches of $(f(z)/z)^\alpha$ and $(g'(w))^\beta$ such that $(f'(0))^\alpha=1=(g'(0))^\beta$. In other words, the above operators are related by
	\begin{equation}\label{4eq1.2}
		C_{\alpha,0}[f,g]\equiv J_{\alpha,0}[f,g]
		\equiv C_{\alpha,\beta}[f,z]\equiv J_{\alpha,\beta}[f,z]
	\end{equation}
	and
\begin{equation}\label{4eq1.3}
C_{0,\beta}[f,g]\equiv I_{0,\beta}[f,g]\equiv 
C_{\alpha,\beta}[z,g]\equiv I_{\alpha,\beta}[z,g].
\end{equation}
The operator $C_{\alpha,\beta}$ can be easily obtain by the Hornich sum of the operators $J_\alpha$ and $I_\beta$ as
$$
C_{\alpha,\beta}[f,g](z)= \big( J_\alpha[f] \oplus I_\beta [g]\big)(z).
$$
The operator $C_{\alpha,\beta}$ contains several well-known operators, simultaneously. Also, we can obtain many known results with the help of this operator $C_{\alpha,\beta}$. For $f=g$, certain geometric properties of $C_{\alpha,\beta}$ have been studied in \cite{Dorff05,Frasin11,Godula79,GMBook}.
	
The organization of this paper as follows:
throughout the paper we assume $-\pi/2<\theta<\pi/2$, $\gamma>0$ and $\lambda<1$.
In the second section, we compute the sets $A(\mathcal{G}(\gamma))$, $A\big(J(\mathcal{S}_\theta(\gamma))\big)$ and $A\big(J(\mathcal{S}(\gamma))\big)$. Also, in the same section we give a restriction on $\theta$ under which the class $J(\mathcal{S}_\theta(\gamma))$ becomes a subclass of the class $\mathcal{S}$. The third section contains several results concerning the operator $C_{\alpha,\beta}$, and their important consequences. 
In the final section, we estimate the sharp bound of pre-Schwarzian norm of range set $J(\mathcal{S}_\theta(\gamma))$ and concludes that every function in the class $J(\mathcal{S}_\theta(\gamma))$, $-\pi/2<\theta<\pi/2$ and $0<\gamma<1$, is a bounded function.
	
\section{The Alexander transform}
	
We begin with the following important lemma proved in \cite{Ali18}, for $0<\gamma\leq 1$. 
\begin{lemma}\label{4lemma2.10}
Let $0<\gamma$. Then $\mathcal{G}(\gamma)=(-\gamma/2)\star \mathcal{K}$.
\end{lemma}
\begin{proof}
Let the mapping $\xi: \mathfrak{F} \longrightarrow \mathfrak{F}$ be defined as $\xi(f)=(-\gamma/2)\star f$, where  $\gamma>0$. Suppose $\xi(f)=g$ then it is easy to compute that $g'=(f')^{(-\gamma/2)}$. Thus we obtain
$$
\frac{zg''(z)}{g'(z)}+1=-\frac{\gamma}{2}\bigg[\frac{zf''(z)}{f'(z)}+1\bigg]+\frac{\gamma}{2}+1.
$$
It follows that $f\in \mathcal{K}$ if and only if $g\in \mathcal{G}(\gamma)$, which leads to the fact that $\xi(\mathcal{K})=\mathcal{G}(\gamma)$.
\end{proof}
For $z,w\in\mathbb{C}$, the line segment joining $z$ and $w$  denote by $[z,w]$.
The proof of the following theorem provides by Aksent'ev and Nezhmetdinov \cite{Aksent'ev82} and Kim et al. \cite{Kimponnusamy04}. 
\begin{theorem}\label{5theorem2.2}
We have the set $A(\mathcal{K})=\{\alpha\in\mathbb{C}:\,|\alpha|\le 1/2\}\cup[1/2,3/2]$.
\end{theorem}
By using Lemma \ref{4lemma2.10} and Theorem \ref{5theorem2.2} we conclude the following result. 
\begin{theorem}\label{5theorem2.3}
Let $0<\gamma$. Then the set $A(\mathcal{G(\gamma)})=\{\alpha\in\mathbb{C}:\,|\alpha|\le 1/\gamma\}\cup[-3/\gamma,-1/\gamma]$.
\end{theorem}
\begin{proof}
The Lemma \ref{4lemma2.10} provides that for every $f\in \mathcal{G}(\gamma)$, $\gamma>0$, there exists a function $g \in \mathcal{K}$ such that $f(z)=((-\gamma/2)\star g)(z)$.  Then we obtain that $I_\alpha[f]=I_{-\gamma\alpha/2}[g]$ for a function $g\in\mathcal{K}$. The final answer provides by Theorem \ref{5theorem2.2}.
\end{proof}
To conclude our next main result and its consequences we need the following lemma.
\begin{lemma}\label{4lemma2.4}
For $-\pi/2<\theta <\pi/2$ and $0<\gamma$, we have
$$
J(\mathcal{S}_{\theta}(\gamma))=I_{e^{-i\theta}\cos\theta}(\mathcal{G}(\gamma)).
$$
\end{lemma}
\begin{proof}
From the expression
\begin{equation}\label{5eq2.1}
\frac{1}{\cos\theta}\Bigg[e^{i\theta}\bigg(\frac{zf''(z)}{f'(z)}+1\bigg)-i \sin\theta\Bigg] = 1+\frac{zk''(z)}{k'(z)}
\end{equation}
and for $J[g]=f$ we can easily observe that $g\in \mathcal{S}_{\theta}(\gamma)$ if and only if $k\in \mathcal{G}(\gamma)$.
		
After simplification of \eqref{5eq2.1} we obtain that 
$$
\frac{f''(z)}{f'(z)}=e^{-i\theta}\cos\theta \frac{k''(z)}{k'(z)},
$$
which gives $f=I_{e^{-i\theta}\cos \theta}[k]$. Hence the proof is complete.
\end{proof}
The above Lemma \ref{4lemma2.10} and Theorem \ref{4lemma2.4} leads to the following theorem:
	
\begin{theorem}\label{4theorem2.5}
For $-\pi/2<\theta<\pi/2$ and $0<\gamma$, we have 
$$
A(J(\mathcal{S}_{\theta}(\gamma)))=\{\alpha\in\mathbb{C}: |\alpha|\leq 1/\gamma\cos\theta\}\cup[-3e^{i\theta}/\gamma\cos\theta,-e^{i\theta}/\gamma\cos\theta].
$$
\end{theorem}
\begin{proof} The property $I_\delta I_\alpha= I_{\delta \alpha}$ together with Lemma \ref{4lemma2.10} gives
$$
J_\alpha(\mathcal{S}_{\theta}(\gamma))=	I_\alpha\big(J(\mathcal{S}_{\theta}(\gamma))\big)=I_\alpha I_{e^{-i\theta}\cos\theta}(\mathcal{G}(\gamma))=I_{\alpha e^{-i\theta}\cos\theta}(\mathcal{G}(\gamma)).
$$
Then, by Theorem \ref{5theorem2.2}, $I_{\alpha e^{-i\theta}\cos\theta}(\mathcal{G}(\gamma))\subset \mathcal{S}$ if and only if $|\alpha|\leq 1/\gamma\cos\theta$ or $\alpha\in[-3e^{i\theta}/\gamma\cos\theta,-e^{i\theta}/\gamma\cos\theta]$. This concludes the proof.
\end{proof}
In the next theorem, we find a restriction on $\theta$ for which the image set $J(\mathcal{S}_{\theta}(\gamma))$ in the class $\mathcal{S}$. 
	
\begin{theorem}\label{4theorem2.6}
Let $0<\gamma$. The relation
$$J(\mathcal{S}_{\theta}(\gamma))\subset \mathcal{S}$$
holds precisely for $\cos\theta\leq1/\gamma$.	
\end{theorem}
\begin{proof}
From Theorem \ref{4theorem2.5} we have $J\big(\mathcal{S}_{\theta}(\gamma)\big)\subset \mathcal{S}$ if and only if $1\in A\big(J\big(\mathcal{S}_{\theta}(\gamma))\big)$. This gives that $\cos \theta \leq 1/\gamma$, completing the proof.	
\end{proof}

\begin{remark}
We remark that functions in $\mathcal{S}_{\theta}(\gamma)$, where $0<\gamma\leq 1$ and $-\pi/2<\theta<\pi/2$, may not be univalent but image set of this class under the Alexander transform is a subclass of univalent functions.
\end{remark}
	

\begin{theorem}
For $\gamma>0$, we have
$$
A\big(J(\mathcal{S}(\gamma))\big)=\bigg\lbrace |\alpha|\leq\frac{1}{\gamma}\bigg\rbrace.
$$	
\end{theorem}
\begin{proof}
By the definition of $\mathcal{S}(\gamma)$, we have 
$$
A\big(J(\mathcal{S}(\gamma))\big)=\underset{\theta}{\bigcap}\,	A\big(J\big(\mathcal{S}_{\theta}(\gamma)\big)\big).
$$
Finally, Theorem \ref{4theorem2.5} concludes the theorem. 
\end{proof}
\section{The operator $C_{\alpha,\beta}$}	
The following theorem characterizes the set in  $\alpha\beta$-plane for which the operators $C_{\alpha,\beta}[f,g]$ either belong to $\mathcal{K(\lambda)}$ or $\mathcal{G(\gamma)}$ or $\mathcal{C}$, whenever $f,g\in\mathfrak{F}$. 
	
\begin{theorem}\label{4theorem2.1}
Let $H$ be a set in $\mathbb{R}^2$.
For $\lambda<1$, $\gamma>0$ and $(\alpha,\beta)\in H$, if $C_{\alpha,\beta}[f,g] \in \mathcal{K(\lambda)}$ or $\mathcal{G(\gamma)}$ or $\mathcal{C}$, then $H$ is a convex set.
\end{theorem}
\begin{proof}
Suppose that $(\alpha,\beta)$ is a locus of point on the line segment $[(\alpha_1,\beta_1),(\alpha_2,\beta_2)]$ joining $(\alpha_1,\beta_1)$
		and $(\alpha_2,\beta_2)$.
		Then it is easy to obtain that
		$$
		C_{\alpha,\beta}[f,g](z)=(t\star C_{\alpha_1,\beta_1}[f,g]\oplus (1-t)\star C_{\alpha_2,\beta_2}[f,g])(z),\quad \mbox{ for $0\le t\le 1$.}
		$$
		A simple computation shows that 
		\begin{equation}\label{4eq2.1}
			1+\frac{zC_{\alpha,\beta}[f,g]''(z)}{C_{\alpha,\beta}[f,g]'(z)} =t\bigg\lbrace1+\frac{zC_{\alpha_1,\beta_1}[f,g]''(z)}{C_{\alpha_1,\beta_1}[f,g]'(z)}\bigg\rbrace+(1-t)\bigg\lbrace1+\frac{zC_{\alpha_2,\beta_2}[f,g]''(z)}{C_{\alpha_2,\beta_2}[f,g]'(z)}\bigg\rbrace.
		\end{equation}
		The conditions $C_{\alpha_i,\beta_i}[f,g] \in \mathcal{K(\lambda)}$, $i=1,2$, in \eqref{4eq2.1} give
		$$
		{\rm Re}\,\bigg \lbrace 1+\frac{zC_{\alpha,\beta}[f,g]''(z)}{C_{\alpha,\beta}[f,g]'(z)}\bigg \rbrace>\lambda.
		$$
		so that $C_{\alpha,\beta}[f,g] \in \mathcal{K(\lambda)}$.
		
		Similarly, the assumptions $C_{\alpha_i,\beta_i}[f,g] \in \mathcal{G(\gamma)}$, $i=1,2$, in \eqref{4eq2.1} provide
		$$
		{\rm Re}\,\bigg \lbrace 1+\frac{zC_{\alpha,\beta}[f,g]''(z)}{C_{\alpha,\beta}[f,g]'(z)}\bigg \rbrace<1+\frac{\gamma}{2},
		$$
		which implies that $C_{\alpha,\beta}[f,g] \in \mathcal{G(\gamma)}$.
		
		Finally, if $C_{\alpha_i,\beta_i}[f,g] \in \mathcal{C}$, $i=1,2$, then from \eqref{4eq2.1} we get that
		$$
		\int_{\theta_1}^{\theta_2} {\rm Re}\,\bigg \lbrace 1+\frac{zC_{\alpha,\beta}[f,g]''(z)}{C_{\alpha,\beta}[f,g]'(z)}\bigg \rbrace d\theta > -\pi, \,\,\ z=re^{i\theta}\text{ and } 0\leq\theta_1<\theta_2\leq 2\pi.
		$$
		The Kaplan's theorem gives that $C_{\alpha,\beta}[f,g] \in \mathcal{C}$. Hence concludes the proof. 
	\end{proof}
	
	The following useful lemma is due to Kim and Merkes which is proved in \cite{Kim74}.
	
	\begin{lemma}\label{4lemma2.6}
		Let $\alpha \in \mathbb{R}$. The function $b_\alpha(z)=\int_{0}^{z}(1+t)^\alpha dt \in \mathcal{C}\,( \text{ or }\mathcal{K})$  if and only if $-3\leq\alpha\leq 1$ (if and only if $-2\leq\alpha\leq0$).
	\end{lemma}
	
	One of the main results we obtain for the operator $C_{\alpha,\beta}$ is the following:
	
	\begin{theorem}\label{4theorem2.7}
		Let $f,g\in \mathcal{K}$. Then $C_{\alpha,\beta}[f,g] \in \mathcal{K}$ if and only if $0\leq\alpha, 2\beta, \alpha+2\beta\leq 2$. 
	\end{theorem}
	\begin{proof}
		It is easy to calculate that 
		\begin{equation}\label{4eq2.2}
			{\rm Re}\,\bigg \lbrace 1+\frac{zC_{\alpha,\beta}[f,g]''(z)}{C_{\alpha,\beta}[f,g]'(z)}\bigg \rbrace =\alpha {\rm Re}\,\bigg\lbrace\frac{zf'(z)}{f(z)}\bigg\rbrace + \beta \,{\rm Re}\,\bigg\lbrace\frac{zg''(z)}{g'(z)}+1 \bigg\rbrace+(1-\alpha-\beta).
		\end{equation}
		It is known that 
		$$
		{\rm Re}\,\bigg\lbrace\frac{zf'(z)}{f(z)}\bigg\rbrace>\frac{1}{2}
		$$
		for $f\in \mathcal{K}$.
		Then, for $\alpha\geq 0$ and $\beta\geq 0$
		$$
		{\rm Re}\,\bigg \lbrace 1+\frac{zC_{\alpha,\beta}[f,g]''(z)}{C_{\alpha,\beta}[f,g]'(z)}\bigg \rbrace>0
		$$
		if $\alpha+2\beta\leq 2$.
		
		For the only if part, we take $f(z)=z$ and $g(z)=z/(1+z)$. Then by Lemma \ref{4lemma2.6} we have $C_{\alpha,\beta}[f,g] \in \mathcal{K}$ if and only if $0\leq\beta\leq 1$. For  $g(z)=z$ and $f(z)=z/(1+z)$,  $C_{\alpha,\beta}[f,g] \in \mathcal{K}$ if and only if $0\leq\alpha\leq 2$, which can easily be verified by using Lemma \ref{4lemma2.6}. Finally, if we choose $f(z)=z/(1+z)$ and $g(z)=z/(1+z)$ then $C_{\alpha,\beta}[f,g] \in \mathcal{K}$ if and only if $0\leq\alpha+2\beta\leq 2$, which also follows by using Lemma \ref{4lemma2.6}. This completes the proof.
	\end{proof}
	
	The relation \eqref{4eq1.2} obtains the following result:
	
	\begin{corollary}
		Let $f\in \mathcal{K}$. Then $J_{\alpha}[f] \in \mathcal{K}$ if and only if $0\leq \alpha\leq 2$. 
	\end{corollary}
	
	The relation \eqref{4eq1.3} leads to the result of Kim and Srivastava \cite[Theorem 2]{KS06} also stated as follows:
	
	\begin{corollary}
		Let $g\in \mathcal{K}$. Then $I_{\beta}[g] \in \mathcal{K}$ if and only if $0\leq\beta\leq 1$. 
	\end{corollary}
	
	The following lemma is discussed in \cite{Kim74}.
	
	\begin{lemma}\label{4lemma2.8}
		If $f\in \mathcal{K}$, then, for $0\leq r<1$, $0\leq \theta_1<\theta_2\leq 2\pi$, we have
		$$
		\frac{\theta_2 - \theta_1}{2} < \int_{\theta_1}^{\theta_2} {\rm Re}\,\bigg(\frac{zf'(z)}{f(z)}\bigg) d\theta \leq \pi + \frac{\theta_2 - \theta_1}{2},
		$$
		and
		$$
		0< \int_{\theta_1}^{\theta_2}  {\rm Re}\,\bigg(\frac{zf''(z)}{f'(z)}+1\bigg)d\theta \leq 2\pi,
		$$
		where $z=r e^{i\theta}$.
	\end{lemma}
	
	In the next theorem we obtain the region in $\alpha \beta$-plane in which $C_{\alpha,\beta}[f,g]\in \mathcal{C}$ whenever $f,g\in \mathcal{K}$.
	
	\begin{theorem}\label{4theorem2.9}
		Let $f,g\in \mathcal{K}$. Then $C_{\alpha,\beta}[f,g] \in \mathcal{C}$ if and only if  $-1\leq \alpha, 2\beta, \alpha+2\beta\leq 3$.
	\end{theorem}
	\begin{proof}
		To obtain the region in the $\alpha\beta$-plane for which $C_{\alpha,\beta}[f,g] \in \mathcal{C}$, whenever $f,g\in \mathcal{K}$, we need to consider four cases on $\alpha$ and $\beta$.
		
		{\bf Case (i) $\alpha\geq 0$ and $\beta \geq 0$:}
		
		Given that $f,g\in \mathcal{K}$. Then by Lemma \ref{4lemma2.8} together with \eqref{4eq2.2} we obtain that 
		$$
		\int_{\theta_1}^{\theta_2} {\rm Re}\,\bigg \lbrace 1+\frac{zC_{\alpha,\beta}[f,g]''(z)}{C_{\alpha,\beta}[f,g]'(z)}\bigg \rbrace d\theta > (1-\frac{\alpha}{2}-\beta)(\theta_2-\theta_1).
		$$ 
		Then by the Kaplan's theorem $C_{\alpha,\beta}[f,g]\in \mathcal{C}$ if $\alpha+2\beta\leq 3$.
		
		{\bf Case (ii) $\alpha\geq 0$ and $\beta < 0$:}
		
		For $f,g\in \mathcal{K}$, Lemma \ref{4lemma2.8} gives that 
		$$
		\int_{\theta_1}^{\theta_2} {\rm Re}\,\bigg \lbrace 1+\frac{zC_{\alpha,\beta}[f,g]''(z)}{C_{\alpha,\beta}[f,g]'(z)}\bigg \rbrace d\theta > (1-\frac{\alpha}{2}-\beta)(\theta_2-\theta_1)+2\pi\beta.
		$$ 
		Then Kaplan's theorem concludes that $C_{\alpha,\beta}[f,g]\in \mathcal{C}$ if $\alpha\leq 3$ and $\beta\geq -1/2$.
		
		{\bf Case (iii) $\alpha< 0$ and $\beta \geq 0$:}
		
		By the assumption on $f,g\in \mathcal{K}$ we estimate   
		$$
		\int_{\theta_1}^{\theta_2} {\rm Re}\,\bigg \lbrace 1+\frac{zC_{\alpha,\beta}[f,g]''(z)}{C_{\alpha,\beta}[f,g]'(z)}\bigg \rbrace d\theta > \alpha\pi+(1-\frac{\alpha}{2}-\beta)(\theta_2-\theta_1)
		$$
		with the help of Lemma \ref{4lemma2.8}. 
		Then $C_{\alpha,\beta}[f,g]\in \mathcal{C}$  if $\alpha\geq -1$ and $\beta\leq 3/2$ by using the Kaplan's theorem.
		
		{\bf Case (iv) $\alpha< 0$ and $\beta< 0$:}
		
		We derive the inequality
		$$
		\int_{\theta_1}^{\theta_2} {\rm Re}\,\bigg \lbrace 1+\frac{zC_{\alpha,\beta}[f,g]''(z)}{C_{\alpha,\beta}[f,g]'(z)}\bigg \rbrace d\theta > (\alpha+2\beta)\pi
		$$
		by using Lemma \ref{4lemma2.8}. 
		Now, the Kaplan's theorem gives that $C_{\alpha,\beta}[f,g]\in \mathcal{C}$ if $\alpha+2\beta\geq -1$.
		
		To prove the sharpness of the result, on the one side we consider the functions $f(z)=z$ and $g(z)=z/(1+z)$. Now by Lemma \ref{4lemma2.6} we have $C_{\alpha,\beta}[f,g] \in \mathcal{C}$ if and only if $-1/2 \leq \beta \leq 3/2$. On the other side, we choose $g(z)=z$ and $f(z)=z/(1+z)$ to verify that  $C_{\alpha,\beta}[f,g] \in \mathcal{C}$ if and only if $-1\leq\alpha\leq3$ which is due to Lemma \ref{4lemma2.6}. Finally, if we choose $f(z)=z/(1+z)$ and $g(z)=z/(1+z)$ then $C_{\alpha,\beta}[f,g] \in \mathcal{C}$ if and only if $-1\leq \alpha+2\beta\leq 3$, which follows from Lemma \ref{4lemma2.6}. This concludes the proof.
	\end{proof}
	
	Due to the relation \eqref{4eq1.2}, we have the following consequence of either Theorem \ref{4theorem2.9} or Theorem B:
	
	\begin{corollary}\cite[Theorem~2]{Merkes71}
		Let $f\in \mathcal{K}$. Then $J_{\alpha}[f] \in \mathcal{C}$ if and only if $-1\leq \alpha \leq 3$.
	\end{corollary}
	
	The relation \eqref{4eq1.3} produces the following corollary as a consequence of either Theorem \ref{4theorem2.9} or Theorem A: 
	
	\begin{corollary}\label{4cor2.7}
		Let $g\in \mathcal{K}$. Then $I_{\beta}[g] \in \mathcal{C}$ if and only if $-1/2\leq \beta \leq 3/2$. 
	\end{corollary}
	It is appropriate to remark here that Corollary~\ref{4cor2.7} can also be deduced from \cite[Theorem 1]{Merkes71} with the help of the classical Alexander theorem.
	
	The next lemma provides that every function in the class $\mathcal{G}(\gamma)$, $0<\gamma$, can be recovered from a function in the class $\mathcal{K}$ by the Hornich multiplication operation, which is already studied in \cite{Ali18} for the limited range $0<\gamma\leq 1$.

The following lemma is observed by Koepf \cite{Koe85}.
\begin{lemma}\label{4lemma2.11}
For all $\lambda<1$, we have $\mathcal{K}(\lambda)=(1-\lambda)\star\mathcal{K}$.
\end{lemma}
	
The following two theorems are natural generalizations of Theorem \ref{4theorem2.7} and  Theorem \ref{4theorem2.9}.	\begin{theorem}\label{4theorem1.8}
For $\lambda<1$, let $f\in \mathcal{K}$ and $g\in \mathcal{K}(\lambda)$ then we have
\begin{enumerate}
\item[\bf (i)] $C_{\alpha,\beta}[f,g]\in  \mathcal{K}(\lambda)$ if and only if $0\leq\alpha, 2\beta(1-\lambda), \alpha+2\beta(1-\lambda)\leq 2(1-\lambda)$.
\item[\bf (ii)] $C_{\alpha,\beta}[f,g]\in  \mathcal{G}(\gamma)$, $\gamma>0$, if and only if $-\gamma\leq\alpha, 2\beta(1-\lambda), \alpha+2\beta(1-\lambda)\leq 0$.
\item[\bf (iii)] $C_{\alpha,\beta}[f,g]\in  \mathcal{C}$ if and only if $-1\leq \alpha, 2\beta(1-\lambda), \alpha+2\beta(1-\lambda)\leq 3$.
\end{enumerate}
\end{theorem}
\begin{proof}
{\bf (i)} Given that $g\in \mathcal{K}(\lambda)$. Then by Lemma \ref{4lemma2.11} there exists a function $h\in\mathcal{K}$ such that $g(z)=((1-\lambda)\star h)(z)=I_{1-\lambda}[h](z)$, which implies that $C_{\alpha,\beta}[f,g](z)=C_{\alpha,\beta(1-\lambda)}[f,h](z)$. Then, we have 
$$
C_{\alpha,\beta}[f,g]=C_{\alpha,\beta(1-\lambda)}[f,h]\in \mathcal{K}(\lambda)=(1-\lambda)\star\mathcal{K}.
$$ 
It is easy to obtain that $(1/(1-\lambda))\star C_{\alpha,\beta(1-\lambda)}[f,h]=C_{\alpha/(1-\lambda),\beta}[f,h]\in \mathcal{K}$. Remaining work can be completed by using Theorem \ref{4theorem2.7}.
		
{\bf (ii)} From part (i) we obtain $C_{\alpha,\beta}[f,g](z)=C_{\alpha,\beta(1-\lambda)}[f,h](z)$. By using Lemma \ref{4lemma2.10}, we observe that $C_{\alpha,\beta(1-\lambda)}[f,h]\in \mathcal{G}(\gamma)=(-\gamma/2)\star \mathcal{K}$. A simple computation provides us $(-2/\gamma)\star C_{\alpha,\beta(1-\lambda)}[f,h]=C_{-2\alpha/\gamma,-2\beta(1-\lambda)/\gamma}[f,h]\in \mathcal{K}$. Now, one can find the desired restrictions on $\alpha$ and $\beta$ by using Theorem \ref{4theorem2.7}. 
		
{\bf (iii)} As we know from part (i) that $C_{\alpha,\beta}[f,g](z)=C_{\alpha,\beta(1-\lambda)}[f,h](z)$, the rest of the steps of the proof follow from Theorem \ref{4theorem2.9}. This completes the proof.
\end{proof}
\begin{theorem}\label{4theorem1.9}
For $\gamma>0$, let $f\in \mathcal{K}$ and $g\in \mathcal{G}(\gamma)$ then we have
\begin{enumerate}
\item[\bf (i)] $C_{\alpha,\beta}[f,g]\in  \mathcal{K}(\lambda)$ if and only if $0\leq\alpha, -\beta\gamma, \alpha-\beta\gamma\leq 2(1-\lambda)$.
\item[\bf (ii)] $C_{\alpha,\beta}[f,g]\in  \mathcal{G}(\gamma)$ if and only if $-\gamma\leq\alpha, -\beta\gamma, \alpha-\beta\gamma\leq 0$.
\item[\bf (iii)] $C_{\alpha,\beta}[f,g]\in  \mathcal{C}$ if and only if $-1\leq \alpha, -\beta\gamma, \alpha-\beta\gamma\leq 3$.
\end{enumerate}
\end{theorem}
\begin{proof}
The proof of part (i), (ii) and (iii) follows from the proof of corresponding part of Theorem \ref{4theorem1.8} by taking $g\in  \mathcal{G}(\gamma)$ instead of $\mathcal{K}(\lambda)$ and using Lemma \ref{4lemma2.10}.
\end{proof}

The proof of Theorem \ref{4theorem2.10} is based on the result \cite[Example 1, Equation (16)]{ponnusamy95} of Ponnusamy and Rajasekaran, in which they observed that  
\begin{equation}\label{4eq2.5}
0<{\rm Re}\,\bigg(\frac{zf'(z)}{f(z)}\bigg)<\frac{4}{3}
\end{equation}
for $f\in \mathcal{G}$. We are using this result to prove the following theorem.
	
In Theorem \ref{4theorem2.7} if we choose $f\in \mathcal{G}$ with the remaining conditions unchanged, then we obtain the following result:
	
\begin{theorem}\label{4theorem2.10}
Let $f\in\mathcal{G}$ and $g\in \mathcal{K}$. Then $C_{\alpha,\beta}[f,g] \in \mathcal{K}$ if and only if  $0\leq\beta\leq 1$, $ \alpha+\beta\leq 1$ and $3\beta-\alpha\leq 3$. 
\end{theorem}
\begin{proof}
The given hypothesis $g\in \mathcal{K}$ along with the relation \eqref{4eq2.2} provides us
\begin{equation}\label{4eq2.6}
{\rm Re}\,\bigg(1+\frac{zC_{\alpha,\beta}[f,g][f]''(z)}{C_{\alpha,\beta}[f,g]'(z)}\bigg) > \alpha{\rm Re}\,\bigg(\frac{zf'(z)}{f(z)}\bigg)+1-\alpha-\beta,
\end{equation}
for $\beta\geq 0$. 
		
If $\alpha\geq0$ then from \eqref{4eq2.5} and \eqref{4eq2.6} we have
$$	
{\rm Re}\,\bigg(1+\frac{zC_{\alpha,\beta}[f,g]''(z)}{C_{\alpha,\beta}[f,g]'(z)}\bigg)> 1-\alpha-\beta.
$$
This provides us $C_{\alpha,\beta}[f,g]\in\mathcal{K}$ for $\alpha+\beta\leq 1$.
		
Now, if $\alpha<0$ then again from \eqref{4eq2.5} and \eqref{4eq2.6} we obtain
$$	
{\rm Re}\,\bigg(1+\frac{zC_{\alpha,\beta}[f,g]''(z)}{C_{\alpha,\beta}[f,g]'(z)}\bigg)> \frac{4}{3}\alpha+1-\alpha-\beta=\frac{1}{3}\alpha+1-\beta.
$$
This gives that $C_{\alpha,\beta}[f,g]\in\mathcal{K}$ for $3\beta-\alpha\leq 3$.
		
Now, we show the sharpness of the result. 
For the choices $f(z)=z\in\mathcal{G}$ and $g(z)=z/(1+z)\in \mathcal{K}$, we have $C_{\alpha,\beta}[f,g](z)=\int_{0}^{z}(1+t)^{-2\beta} \in \mathcal{K}$ if and only if $0\leq\beta\leq1$ by using Lemma \ref{4lemma2.6}.

Further, we consider $f(z)=[1-(1-z)^2]/2\in\mathcal{G}$ and $g(z)=z/(1+z)$ then we obtain
$$	
1+\frac{zC_{\alpha,\beta}[f,g]''(z)}{C_{\alpha,\beta}[f,g]'(z)} =\frac{2-(1+\alpha)z}{2-z}-\frac{2\beta z}{1+z}.
$$
For $\alpha+\beta>1$, it is easy to see that $0<2/(\alpha+\beta+1)<1$. So if we choose $z=2/(\alpha+\beta+1)$ then we have
$$	
1+\frac{zC_{\alpha,\beta}[f,g]''(z)}{C_{\alpha,\beta}[f,g]'(z)} =\frac{\beta}{\alpha+\beta}-\frac{4\beta}{\alpha+\beta+3}=\frac{-3\beta(\alpha+\beta-1)}{(\alpha+\beta)(\alpha+\beta+3)}.
$$
The above calculation shows that 
$$
{\rm Re}\,\bigg(1+\frac{zC_{\alpha,\beta}[f,g]''(z)}{C_{\alpha,\beta}[f,g]'(z)}\bigg)<0,
$$
for $z=2/(\alpha+\beta+1)$. This implies that $C_{\alpha,\beta}[f,g]\notin\mathcal{K}$ for $\alpha+\beta>1$.
		
To complete our proof,  we choose $f(z)=[(1+z)^2-1]/2\in\mathcal{G}$ and $g(z)=z/(1+z)\in \mathcal{K}$. Then we get
$$	
1+\frac{zC_{\alpha,\beta}[f,g]''(z)}{C_{\alpha,\beta}[f,g]'(z)}=\frac{2+(1+\alpha)z}{2+z}-\frac{2\beta z}{1+z}.
$$
For $z=2/(3\beta-\alpha-1)$, we obtain
$$
{\rm Re}\,\bigg(1+\frac{zC_{\alpha,\beta}[f,g]''(z)}{C_{\alpha,\beta}[f,g]'(z)}\bigg)=\frac{3\beta}{3\beta-\alpha}-\frac{4\beta}{3\beta-\alpha+1}=\frac{\beta(\alpha-3\beta+3)}{(3\beta-\alpha)(3\beta-\alpha+1)}<0,
$$  
for $3\beta-\alpha-3>0$, $0<z<1$. It concludes that $C_{\alpha,\beta}[f,g]\notin\mathcal{K}$ for $3\beta-\alpha-3>0$,
completing the proof. 
\end{proof}
The following two consecutive theorems are the consequences of Theorem \ref{4theorem2.10}. We can obtain the proof of these theorems by a  similar process, which we are using in Theorem \ref{4theorem1.8} with the help of  Theorem \ref{4theorem2.10}.
\begin{theorem}\label{4theorem1.10}
For $\lambda<1$, let $f\in\mathcal{G}$ and $g\in \mathcal{K}(\lambda)$ then we have
\begin{enumerate}
\item[\bf (i)]$C_{\alpha,\beta}[f,g]\in  \mathcal{K}(\lambda)$ if and only if $-3(1-\lambda)\leq\alpha\leq (1-\lambda)$, $0\leq\beta\leq 1$, $ \alpha+\beta(1-\lambda)\leq (1-\lambda)$ and $3\beta(1-\lambda)-\alpha\leq 3(1-\lambda)$.
\item[\bf (ii)] $C_{\alpha,\beta}[f,g]\in  \mathcal{G}(\gamma)$, $\gamma>0$, if and only if $-\gamma/2\leq\alpha\leq 3\gamma/2$, $-\gamma/2\leq\beta(1-\lambda)\leq 0$, $ \alpha+\beta(1-\lambda)\geq -\gamma/2$ and $3\beta(1-\lambda)-\alpha\geq -\gamma/2$.
\end{enumerate}
\end{theorem}
\begin{theorem}\label{4theorem1.11}
For $\gamma>0$, let $f\in\mathcal{G}$ and $g\in \mathcal{G}(\gamma)$ then we have
\begin{enumerate}
\item[\bf (i)]$C_{\alpha,\beta}[f,g]\in \mathcal{K}(\lambda)$, $\lambda<1$ if and only if  $-3(1-\lambda)\leq\alpha\leq (1-\lambda)$, $-2(1-\lambda)/\gamma\leq\beta\leq 0$, $ \alpha-\beta\gamma/2\leq (1-\lambda)$ and $-3\beta\gamma/2-\alpha\leq 3(1-\lambda)$.
\item[\bf (ii)] $C_{\alpha,\beta}[f,g]\in\mathcal{G}(\gamma)$ if and only if $-\gamma/2\leq\alpha\leq 3\gamma/2$, $0\leq\beta\leq 1$, $ -2\alpha/\gamma+\beta\leq 1$ and $3\beta+2\alpha/\gamma\leq 3$.
\end{enumerate}
\end{theorem}
	
In view of the relation \eqref{4eq1.2}, Theorem \ref{4theorem1.10}(i) or Theorem \ref{4theorem1.11}(i) obtain the following corollary which may be of independent interest.
	
\begin{corollary}\label{4corollary1.13}
For $\alpha\in\mathbb{R}$ and $f\in\mathcal{G}$, $J_\alpha[f] \in \mathcal{K}(\lambda)$ if and only if $-3(1-\lambda)\leq\alpha\leq (1-\lambda)$.
\end{corollary}
With the help of \eqref{4eq1.2}, Theorem \ref{4theorem1.10}(ii) or in Theorem \ref{4theorem1.11}(ii) leads to the following  corollary.
\begin{corollary}\label{4corollary1.14}
For $\alpha\in\mathbb{R}$ and $f\in\mathcal{G}$, $J_\alpha[f] \in \mathcal{G}(\gamma)$ if and only if $-\gamma/2\leq\alpha\leq 3\gamma/2$.
\end{corollary}
	

\section{Pre-Schwarzian Norms}
The pre-Schwarzian norm of a function $f\in \mathfrak{F}$ is defined as
$$
\|f\|=\sup_{z\in\mathbb{D}}\,(1-|z|^2)\left|\frac{f''(z)}{f'(z)}\right|
$$
where the quantity $f''/f'$ is often referred to as the pre-Schwarzian derivative of $f$. It is well-known that $\|f\|\le 6$ for $f\in\mathcal{S}$ as well as for $f\in\mathcal{S}^*$ and, 
The sharp estimation $\|f\|\le 4$, for $f\in\mathcal{K}$, was later generalized by Yamashita \cite{Yamashita99} to the class $\mathcal{K}(\lambda)$, $0\le \lambda<1$. Recently, in \cite{Ali18}, Yamashita's result has been further extended to
$\mathcal{K}(\lambda)$, $-1/2\le \lambda<1$. However, Kumar and Sahoo in \cite{Shankey19} proved that the result of Yamashita holds true for all $\lambda<1$. Furthermore, $\|f\|\leq 2\gamma$ for $f\in \mathcal{G}(\gamma)$, $0<\gamma\leq1$, is obtained in \cite{Ali18}, and it is also true for $\gamma>0$ explain in the next theorem. 

\begin{theorem}\label{4theorem4.1}
For $\gamma>0$, if  $f\in \mathcal{G}(\gamma)$ then $\|f\|\leq 2\gamma$ and the bound is sharp.
\end{theorem}
\begin{proof}
As observed in Lemma \ref{4lemma2.10}, the functions in $\mathcal{G}(\gamma)$ can be expanded in terms of the Hornich scalar multiplication: 
$(-\gamma/2)\star \mathcal{K}=\{(-\gamma/2)\star g : g\in \mathcal{K} \}$. Then for $f\in \mathcal{G}(\gamma)$ there exist a $g\in \mathcal{K}$ such that $f=(-\gamma/2)\star g$. Therefore, $\|f\|=(\gamma/2)\|g\|\leq 2\gamma$. We can easily obtain the sharpness by considering the example 
$$
g(z)=\frac{1}{1+\gamma}[1-(1-z)^{1+\gamma}], \,\ z\in\mathbb{D}.
$$
Hence conclude.
\end{proof}
The next theorem is a consequence of Lemma \ref{4lemma2.4} and Theorem \ref{4theorem4.1}.
\begin{theorem}\label{4theorem4.2}
For each $\theta\in(-\pi/2,\pi/2)$ and $\gamma>0$, the sharp inequality $\|f\|\leq 2\gamma\cos\theta$ holds for $f\in J(\mathcal{S}_{\theta}(\gamma))$.
\end{theorem}
\begin{proof}
It is easy to calculate that $\|I_\lambda(f)\|=|\lambda|\|f\|$. 
Secondly, By Lemma \ref{4lemma2.4} for $f \in J(\mathcal{S}_{\theta}(\gamma))$ there exists a function $k \in \mathcal{G}(\gamma)$ such that $f=I_{e^{-i\theta}\cos\theta}[k]$. Then by using Theorem \ref{4theorem4.1} we obtained that $\|f\|=|\cos \theta|\|k\|\leq2\gamma\cos \theta$.

To conclude the final part, we consider the function 
$$
g_\theta(z)=z(1-z)^{\gamma e^{-i\theta}\cos\theta}\in \mathcal{S}_{\theta}(\gamma),
$$
where $\theta\in(-\pi/2,\pi/2)$ and $\gamma>0$. After applying the Alexander transform over the function $g_\theta$, we can easily obtained that $\|J[g_\theta]\|=2\gamma\cos\theta$.
This proves the sharpness part.	
\end{proof}
\begin{remark} It is proved in \cite{Kim02} that $f\in \mathfrak{F}$ is bounded if $\|f\|<2$ and bound depends only on the value of $\|f\|$. 
From Theorem \ref{4theorem4.2}, we can easily see that for every $\theta\in(-\pi/2,\pi/2)$ and $0<\gamma<1$, the inequality $\|f\|< 2$ holds for $f\in J(\mathcal{S}_{\theta}(\gamma))$. Hence, a function in the class  $J(\mathcal{S}_{\theta}(\gamma))$, where $\theta\in(-\pi/2,\pi/2)$ and $0<\gamma<1$,  is bounded by a constant depending on $\theta$ and $\gamma$.
\end{remark}

\medskip
\noindent
{\bf Conflict of Interests.} Author declare that there is no conflict of interests 
regarding the publication of this paper.
	
\bigskip
\noindent
{\bf Acknowledgment.} I would like to thank my Ph.D. supervisor Prof. Swadesh Kumar Sahoo for his helpful remarks.
The work of the author is supported by CSIR, New Delhi (Grant No: 09/1022(0034)/2017-EMR-I).

\end{document}